\documentclass[11pt]{amsart}
\usepackage{geometry}
\usepackage{hyperref}
\usepackage{MnSymbol}
\usepackage{mathrsfs}
\usepackage{amsmath,amsthm}
\usepackage{tikz}
\usetikzlibrary{
  knots,
  hobby,
  decorations.pathreplacing,
  shapes.geometric,
  calc,
  decorations.markings
}
\usepgfmodule{decorations}
\usepackage{float}
\usepackage{braids}
\usepackage{setspace}
\usepackage{tikz-cd}

\usepackage{wrapfig}
\usepackage{lipsum} 

\theoremstyle{plain}
\theoremstyle{definition}
\newtheorem{Theorem}{Theorem}[section]

\newtheorem{conjecture}[Theorem]{Conjecture}
\newtheorem{Lemma}[Theorem]{Lemma}
\newtheorem{Proposition}[Theorem]{Proposition}

\newtheorem{definition}[Theorem]{Definition}
\newtheorem{Remark}[Theorem]{Remark}

\numberwithin{equation}{section}
\DeclareMathAlphabet\mathbb{U}{msb}{m}{n}

\begin{document}

\title{On minimal crossing number braid diagrams and homogeneous braids}

\author{Ilya Alekseev}
\author{Geidar Mamedov}
\thanks{This work was supported by the Russian Science Foundation \mbox{(project no. 19-11-00151)}.}

\address{Laboratory of Modern Algebra and Applications, St. Petersburg State University, 14th Line, 29b, Saint Petersburg, 199178 Russia}
\email{ilyaalekseev@yahoo.com}
\address{Laboratory of Continuous Mathematical Education (School 564 of St. Petersburg), nab. Obvodnogo kanala 143, Saint Petersburg, Russia}
\email{geidar.mamedov1@gmail.com}

\begin{abstract}
We study braid diagrams with a minimal number of crossings. Such braid diagrams correspond to geodesic words for the braid groups with standard Artin generators. We prove that a diagram of a homogeneous braid is minimal if and only if it is homogeneous. We conjecture that monoids of homogeneous braids are Artin-Tits monoids and prove that monoids of alternating braids are right-angled Artin monoids. Using this, we give a lower bound on the growth rate of the braid groups.
\end{abstract}

\maketitle

\section{Introduction}

In this paper, we study the braid groups $B_n$. These groups have the following presentations with standard Artin generators:
\begin{align*}
B_n = \langle \sigma_1, \ldots, \sigma_{n-1} \mid \sigma_k\sigma_{k+1}\sigma_k=\sigma_{k+1}\sigma_k\sigma_{k+1},\ 1 \leq k \leq n-2, \ \ \  \sigma_i\sigma_j=\sigma_j\sigma_i, \ |i-j|\geq 2 \rangle,
\end{align*}
The relation $\sigma_i\sigma_j=\sigma_j\sigma_i$ is usually referred to as the far commutativity relation and $\sigma_k\sigma_{k+1}\sigma_k=\sigma_{k+1}\sigma_k\sigma_{k+1}$ is usually referred to as the braid relation. For example, $B_3 = \langle a,b \mid aba = bab\rangle$. By a braid word or a word in $B_n$ we mean a word in the alphabet $\{\sigma_1, \ldots, \sigma_{n-1},\sigma_1^{-1}, \ldots, \sigma_{n-1}^{-1}\}.$ One may visualize braid words by their (vertical) braid diagrams, such as in Figure \ref{HomBraidExample} and Figure \ref{PNotMinimal}. There is a correspondence between the far commutativity classes of braid words and the plane isotopy classes of braid diagrams. A braid diagram is called minimal (short for minimal crossing number braid diagram) if it has the least possible number of crossings among all diagrams representing the same braid. The number of crossings on a minimal diagram of a braid is called the crossing number of the braid. A braid word is called geodesic if the corresponding braid diagram is minimal. Our aim is to study the language of geodesic words in the braid groups with the standard generating set.

In \cite{Stallings}, J. Stallings introduced the concept of homogeneous braid words. In particular, he proved that the Alexander closure of a non-degenerate homogeneous braid is a fibered link. More precisely, let $e_1, \ldots, e_{n-1} \in \{1,-1\}$. A braid word is said to be $(e_1,\ldots,e_{n-1})$--homogeneous if it contains none of the letters $\sigma_1^{-e_1}, \sigma_2^{-e_2}, \ldots, \sigma_{n-1}^{-e_{n-1}}$. A braid diagram is called positive if it is $(1,1,1, \ldots, 1)$--homogeneous. If a word is $(e_1,\ldots,e_{n-1})$--homogeneous for some $e_1, \ldots, e_{n-1}$, we say that it is homogeneous. In a similar way, we say that a braid diagram is homogeneous (resp. positive) if none of its columns contains both types of crossings (resp. negative crossings), see Figure \ref{HomBraidExample} and Figure \ref{E2}. A braid is called homogeneous (resp. positive) if it has a homogeneous (resp. positive) diagram.

In \cite{Cromwell} the concept of homogeneous braids was generalized to links. The class of homogeneous links includes the classes of alternating links (such as rational links) and positive links. There is a well-known classification of minimal diagrams for alternating links. In contrast to alternating links, there are homogeneous (resp. positive) links that admit non-minimal homogeneous (resp. positive) diagrams and admit non-homogeneous (resp. non-positive) minimal diagrams (for example, $11_{550}$ in the Rolfsen knot table, see \cite{Stoimenow}). We present the following result.

\begin{spacing}{1.5}
\end{spacing}

\noindent {\bf Theorem \ref{HomogeneousAreGeodesic}}. A diagram $D$ of a homogeneous braid is minimal if and only if $D$ is homogeneous.

\begin{spacing}{1.5}
\end{spacing}

The proof is based on the Alexander-Conway polynomial for tangles. In the case of positive braids, the result is well known and easy to prove, see Lemma~\ref{Obviouslemma} below. In \cite{Turaev1990} it was noticed that any reduced alternating tangle diagram (and hence any reduced alternating braid diagram) is minimal.

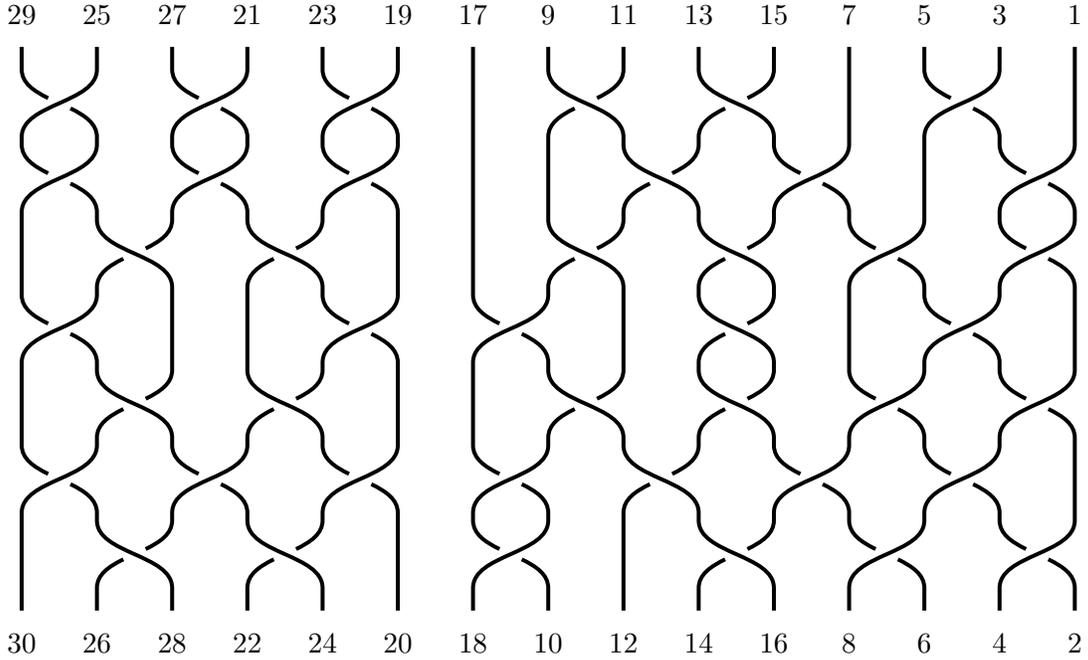
\begin{figure}[H]
\centering
\begin{tikzpicture}
\braid[scale=1,line width=1.5pt] (HomogeneousBraid)
s_1^{-1}-s_3^{-1}-s_5^{-1}-s_8-s_{10}-s_{13}^{-1}
s_1^{-1}-s_3^{-1}-s_5^{-1}-s_9-s_{11}^{-1}-s_{14}^{-1}
s_2-s_4-s_8-s_{10}-s_{12}^{-1}-s_{14}^{-1}
s_1^{-1}-s_5^{-1}-s_7^{-1}-s_{10}-s_{13}^{-1}
s_2-s_4-s_8-s_{10}-s_{12}^{-1}-s_{14}^{-1}
s_1^{-1}-s_3^{-1}-s_5^{-1}-s_7^{-1}-s_9-s_{11}^{-1}-s_{13}^{-1}
s_2-s_4-s_7^{-1}-s_{10}-s_{12}^{-1}-s_{14}^{-1};
\node[label=above:29] at (HomogeneousBraid-1-s) {};
\node[label=above:25] at (HomogeneousBraid-2-s) {};
\node[label=above:27] at (HomogeneousBraid-3-s) {};
\node[label=above:21] at (HomogeneousBraid-4-s) {};
\node[label=above:23] at (HomogeneousBraid-5-s) {};
\node[label=above:19] at (HomogeneousBraid-6-s) {};
\node[label=above:17] at (HomogeneousBraid-7-s) {};
\node[label=above:9] at (HomogeneousBraid-8-s) {};
\node[label=above:11] at (HomogeneousBraid-9-s) {};
\node[label=above:13] at (HomogeneousBraid-10-s) {};
\node[label=above:15] at (HomogeneousBraid-11-s) {};
\node[label=above:7] at (HomogeneousBraid-12-s) {};
\node[label=above:5] at (HomogeneousBraid-13-s) {};
\node[label=above:3] at (HomogeneousBraid-14-s) {};
\node[label=above:1] at (HomogeneousBraid-15-s) {};
\node[label=below:24] at (HomogeneousBraid-1-e) {};
\node[label=below:30] at (HomogeneousBraid-2-e) {};
\node[label=below:28] at (HomogeneousBraid-3-e) {};
\node[label=below:22] at (HomogeneousBraid-4-e) {};
\node[label=below:20] at (HomogeneousBraid-5-e) {};
\node[label=below:26] at (HomogeneousBraid-6-e) {};
\node[label=below:16] at (HomogeneousBraid-7-e) {};
\node[label=below:6] at (HomogeneousBraid-8-e) {};
\node[label=below:10] at (HomogeneousBraid-9-e) {};
\node[label=below:4] at (HomogeneousBraid-10-e) {};
\node[label=below:18] at (HomogeneousBraid-11-e) {};
\node[label=below:12] at (HomogeneousBraid-12-e) {};
\node[label=below:14] at (HomogeneousBraid-13-e) {};
\node[label=below:2] at (HomogeneousBraid-14-e) {};
\node[label=below:8] at (HomogeneousBraid-15-e) {};
\end{tikzpicture}
\caption{An example of degenerate $(1,-1,1,-1,1,1,1,-1,-1,-1,1,1,1,1)$-homogeneous braid equipped with a homogeneous ordering.}\label{HomBraidExample}
\end{figure}

Note that $(e_1, \ldots, e_{n-1})$--homogeneous braids form a submonoid $\mathcal{HBM}(e_1, \ldots, e_{n-1})$ in $B_n$. These submonoids have generators $\sigma_1^{e_1}, \ldots, \sigma_{n-1}^{e_{n-1}}$, which we call the standard generators.

As a generalization of the previous result, we propose the following conjecture, which states that $\mathcal{HBM}(e_1, \ldots, e_{n-1})$ are Artin-Tits monoids.
\begin{conjecture}\label{HBGR} 
Let $e_1, \ldots, e_{n-1} \in \{1,-1\}$. Denote by $a_1, \ldots, a_{n-1}$ the standard generators of the monoid $\mathcal{HBM}(e_1, \ldots, e_{n-1})$. Then the monoid has a presentation with relations
\begin{enumerate}
\item $a_ka_{k+1}a_k = a_{k+1}a_ka_{k+1}$ whenever $e_k = e_{k+1}$;
\item $a_ia_j = a_ja_i$, $|i-j| \geq 2$.
\end{enumerate}
\end{conjecture}

Denote by $B_n^+ := \mathcal{HBM}(1,1,\ldots,1)$ the monoid of positive braids. In the particular case of positive braids, the conjecture above is true (see e.g. Theorem 9.2.5 in \cite{WPG}), that is,
\begin{align*}
B_n^+ = \langle \sigma_1, \ldots, \sigma_{n-1} \mid \sigma_k\sigma_{k+1}\sigma_k=\sigma_{k+1}\sigma_k\sigma_{k+1},\ 1 \leq k \leq n-2, \ \ \  \sigma_i\sigma_j=\sigma_j\sigma_i, \ |i-j|\geq 2 \rangle.
\end{align*}
The proof is based on the existence of so-called greedy normal form for positive braids (see \cite{GarsideBook}).

A braid diagram is called alternating if as we travel through from top to bottom, the strings go through the crossings alternately between overpasses and underpasses. A braid word is called alternating if the corresponding diagram is alternating. A braid is called alternating if it has an alternating diagram. It is easy to see that if a braid word $w$ is either $(1,-1,1,-1,\ldots,(-1)^n)$--homogeneous or $(-1,1,-1,\ldots,(-1)^{n-1})$--homogeneous, then $w$ is alternating. Note that the converse does not hold true in general: the braid word $\sigma_1\sigma_3^{-1}$ is alternating but it is not homogeneous. Nevertheless, if a braid word $w$ is alternating and non-degenerate (that is, the corresponding braid diagram does not split), then it is either $(1,-1,1,-1,\ldots,(-1)^n)$--homogeneous or $(-1,1,-1,\ldots,(-1)^{n-1})$--homogeneous.

We prove the conjecture above for alternating braids. Also, we give a proof of the classification of minimal diagrams for alternating braids using the similar classification for alternating links. The proof of the latter is based on the Jones polynomial techniques.
\begin{spacing}{1.5}
\end{spacing}

\noindent {\bf Theorem \ref{ProofOfConjectureForAlternatingBraids}}. Suppose $w$ is a braid word representing an alternating braid. Then, $w$ is geodesic if and only if $w$ is alternating. Moreover, the monoids $\mathcal{HBM}(1,-1,\ldots,(-1)^n)$ and $\mathcal{HBM}(-1,1,\ldots,(-1)^{n-1})$ are isomorphic and have the following presentation
$$\langle a_1, a_2, \ldots, a_{n-1} \mid a_ia_j=a_ja_i, \ |i-j|\geq 2\rangle.$$

This means that the monoids above are right--angled Artin monoids. The latter is also called graph monoids, Cartier-Foata monoids, (free) partially commutative monoids, trace monoids, semifree monoids, and locally free monoids, see \cite{Charney}.
\begin{figure}[H]
\centering
\begin{tikzpicture}
\braid[scale=0.8,rotate=90,
line width=1.5pt] s_2^{-1}s_1^{-1}s_3^{-1}s_2^{-1}s_2^{-1}s_3^{-1}s_1s_2;
\end{tikzpicture} \ \hspace{1cm}
\begin{tikzpicture}
\braid[scale=0.8,rotate=90,
line width=1.5pt] s_2^{-1}s_3^{-1}s_2^{-1}s_1s_2^{-1}s_3^{-1}s_3^{-1}s_2;
\end{tikzpicture}
\caption{Braid diagrams rotated on 90 degrees that correspond to the braid words $\sigma_2\sigma_1\sigma_3\sigma_2\sigma_2\sigma_3\sigma_1^{-1}\sigma_2^{-1}$ and $\sigma_2\sigma_3\sigma_2\sigma_1^{-1}\sigma_2\sigma_3\sigma_3\sigma_2^{-1}$. Both of them are not minimal.
}\label{PNotMinimal}
\end{figure}
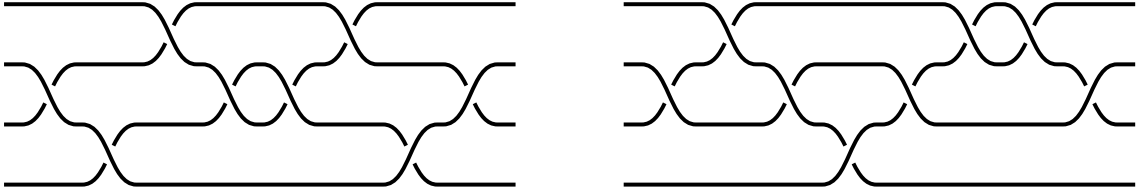

As an application of the results above, we improve known lower bounds on the logarithmic volume of the braid groups as follows. Denote by $P_n(m), A_n(m), H_n(m),$ and $\Gamma_n(m)$ the number of positive, alternating, homogeneous, and all braids with crossing number $m$ in $B_n$ respectively. The function $\Gamma_n(m)$ is called the (spherical) growth function of $B_n$. The limits 
\begin{align*}
v(B_n^+) := \lim\limits_{m \to \infty} \dfrac{\log P_n(m)}{m}, \qquad v(B_n) := \lim\limits_{m \to \infty} \dfrac{\log \Gamma_n(m)}{m}
\end{align*}
are called the logarithmic volume of $B_n^+$ and logarithmic volume of $B_n$ respectively. The results of computations show that $v(B_3) = \log 2$. Theorem 10 in \cite{Vershik} implies that for all sufficiently large $n$
\begin{align*}
\log 2 \leq v(B_n^+) \leq \log 4, \qquad \frac{1}{2} \log 7 \leq v(B_n) \leq \log 7.
\end{align*}
Using Theorem \ref{ProofOfConjectureForAlternatingBraids} and calculations in right--angled Artin monoids from Corollary 2 in \cite{Vershik}, we see that $\log (A_n(m))/m$ converges to $\log 4$ as $m$ tends to infinity (for all sufficiently large $n$). Since $A_n(m) \leq \Gamma_n(m)$, it follows that $\log 4 \leq v(B_n) \leq \log 7$ for all sufficiently large $n$.

In \cite{Sabalka} it is shown that $\sum_{m=0}^\infty \Gamma_3(m) z^m$ is a rational function, that is, $B_3$ has rational growth. It is unknown whether $\sum_{m=0}^\infty \Gamma_n(m) z^m$ is rational for some $n \geq 4$. It is known that there is no linear recurrence relation of degree $\leq 13$ for $\Gamma_4(m)$. In \cite{Deligne,Xu} it is shown that $\sum_{m=0}^\infty P_n(m) z^m$ is rational for all $n \geq 2$ and hence $v(B_n^+)$ could be found explicitly. Actually, as it is shown in \cite{Saito}, any Artin-Tits monoid has rational growth. 

An explicit description of geodesic words in $B_3$ is well known (see \cite{Sabalka,Berger} and Remark \ref{B3} below). Actually, the language of geodesic words in $B_3$ is regular and consists of 27 different cone types. It is an intriguing open problem whether the language of geodesics for $B_n$ (with respect to the standard generators) is regular for some $n\geq 4$, see \cite{Mairesse}. In \cite{CharneyMeier} it is proved that the language of geodesic words in $B_n$ is regular with respect to the generating set of simple divisors of the fundamental (Garside) element. See also \cite{Charney1,Charney2,Dehornoy,Brazil}. There is a problem due to Stallings (Problem 1.8 in Kirby's List \cite{K}) whether geodesic braid words are closed under end extension (replacing a final letter $s$ by $s^n$ for $n \in \mathbb{N}$). Below we present a natural generalizations for the conjecture of Stallings, which we call Stallings' minimal braid diagrams conjecture. A braid $b \in B_n$ is called minimal conjugacy length braid if $b$ has the smallest crossing number among all braids in the conjugacy class $\{c^{-1}bc \mid c \in B_n\}$. A braid word $w$ is called minimal conjugacy length word if $w$ is geodesic and the corresponding braid is minimal conjugacy length braid. An explicit description of minimal conjugacy length words in $B_3$ is well known (see Theorem 6.1. in \cite{AS}).

\begin{conjecture}\label{SMBC}
Let $w,u,v$ be braid words in $B_n$ and $s \in \{\sigma_1, \ldots, \sigma_{n-1},\sigma_1^{-1}, \ldots, \sigma_{n-1}^{-1}\}$.
\begin{enumerate}
\item If $ws$ is geodesic, then $wss$ is geodesic.
\item If $usv$ is geodesic, then $ussv$ is geodesic.
\item If $ws$ is of minimal conjugacy length, then $wss$ is of minimal conjugacy length.
\end{enumerate}
\end{conjecture}

In case $n=2$, this is obvious. In case $n=3$, we see that conjecture is true because of an explicit description of both geodesic and minimal conjugacy length words in $B_3$. In \cite{FourStrandsAlgorithm}, a fast algorithm of searching minimal diagram of a given braid was presented (without a proof). It is easy to check that the correctness of the algorithm is equivalent to the original conjecture of Stallings in case $n=4$.

Since homogeneous braid words are closed under extensions above, Theorem \ref{HomogeneousAreGeodesic} implies that the first two parts of the conjecture above hold true for homogeneous braid words.

The third part of the conjecture above is also known as Question 6.1 in \cite{AS}. We may interpret this question as an analogue of Stallings' conjecture for minimal crossing number diagrams of braided links in a solid torus.
Also, there are natural analogues of Stallings' conjecture for minimal crossing number diagrams of links in the 3-sphere and tangles, which hold for reduced alternating diagrams. 

In addition to the results above, we present several sufficient conditions on a braid diagram to be minimal. Those diagrams that satisfy our conditions are said to be winding. Also, we prove that Stalling's conjecture holds for some of them. With the help of winding diagrams, we give lower bound on the number of minimal braid diagrams with $m$ crossings as follows. Denote by $p_n(m), a_n(m), h_n(m)$, and $\gamma_n(m)$ the number of positive, alternating, homogeneous, and geodesic braid words of length $m$ in $B_n$ respectively. The function $\gamma_n(m)$ is called the (spherical) geodesic growth function of $B_n$. As Lemma \ref{eq} shows, Conjecture \ref{SMBC} \mbox{implies $\gamma_n(m+1) \geq (n-1)\gamma_n(m)$.} Note that $\gamma_n(m)\leq (2n-2)^m = 2^m \cdot (n-1)^m$. 
We present a construction that produces geodesic braid words of a given length in $B_n$. In particular, we have the following result for all large $n$.
\begin{spacing}{1.5}
\end{spacing}

\noindent {\bf Theorem \ref{GeodesicsGenerator}}. For any $k \in \mathbb{N}$ there exist constants $c>0$ and $N \in \mathbb{N}$ such that
\begin{align*}
\gamma_n(m) \geq c\cdot m^k \cdot (n-1)^m
\end{align*}
for all $m \geq N$.
\begin{spacing}{1.5}
\end{spacing}

For instance, this shows that $h_n(m)/\gamma_n(m)$ converges to zero as $m$ tends to infinity. It is unknown whether $\sum_{m=0}^\infty \gamma_n(m)z^m$ is rational for some $n \geq 4$. In \cite{Sabalka} it is shown that $\sum_{m=0}^\infty \gamma_3(m)z^m$ is rational. The results of computations show that $\log(\gamma_3(m))/m$ converges to $\log (1+\sqrt{2})$ as $m$ tends to infinity. It would be interesting to find more accurate lower bound of $\gamma_n(m)/(n-1)^m$. 

\subsection*{Acknowledgment}
The authors are grateful to A. V. Malyutin for useful discussions.

\section{Preliminaries}

The basic notions of braid theory and knot theory (such as knots, links, prime link diagrams etc.) can be found in \cite{Adams,MurasugiB}. We use the language of geometric group theory that is described in \cite{OH}. Given a set of symbols $S$, we introduce the following notation: $S^{-1} = \{s^{-1} \mid s \in S\}$. Denote by $F(S)$ the set of all words in the alphabet $S\cup S^{-1}$. A word $w \in F(S)$ is called freely reduced if for all $s \in S$ the word $w$ does not contain both $s$ and $s^{-1}$. The letters $s \in S$ are called positive while the letters $s^{-1} \in S^{-1}$ are called negative. Given a group $G = \langle S \rangle$ with a generating set $S$, for $w \in F(S)$ we denote by $[w] \in G$ the corresponding element of the group. For a word $w \in F(S)$, we denote by $|w|$ the length of~$w$. A word $w \in F(S)$ is called non-degenerate if for all letters $s \in S$ the word $w$ contains either $s$ or $s^{-1}$, else $w$ is called degenerate. An element $g \in G$ is called degenerate if it has a degenerate word representative, else $g$ is called non-degenerate.

Let $S_n$ be the symmetric group of degree $n$. Consider the homomorphism $B_n \longrightarrow S_n$ given by $\sigma_i \mapsto (i,i+1)$. The image of a braid under this homomorphism is called the permutation corresponding to the braid. A braid is called pure if the corresponding permutation is the identity permutation. 

Given a braid word $v$, denote by $p(v)$ the number of positive letters in $v$ and by $n(v)$ the number of negative ones. Geometrically, $p(v)$ is a number of positive crossings on the braid diagram corresponding to $v$, and $n(v)$ is the number of negative ones. We introduce the notion of a shadow crossing, see Figure \ref{E2}. By a braid shadow, we mean a picture including shadow crossings only, see Figure \ref{W}. For any braid diagram there is the corresponding braid shadow that is defined in the natural way. A braid diagram is called non-degenerate if its shadow is connected. Put ${\rm exp}(v):= p(v) - n(v)$. The latter is an analogue of the writhe in knot theory. It is easy to check from the presentation of $B_n$ that ${\rm exp}(v)$ does not depend on a word $v$ representing a given braid. Note that $|v| = p(v) + n(v).$

\begin{Lemma}\label{Obviouslemma}
Let $w$ be a geodesic braid word and let $u$ be another braid word representing the same braid. One has $|w| \leq |u|$, $p(w) \leq p(u)$, and $n(w) \leq n(w)$. Moreover, in each inequality, equality holds true if and only if $u$ is geodesic.
\end{Lemma}
\begin{proof}
This follows from $|v| = 2p(v) - {\rm exp}(v) = 2n(v) + {\rm exp}(v)$ and the fact that ${\rm exp}(w) = {\rm exp}(u).$
\end{proof}
This result implies that a braid word representing a positive braid is geodesic if and only if it is positive.

\begin{Remark}\label{B3}
We need an explicit description of geodesic words for $B_3$ given in \cite{Sabalka}. Actually, there is a misprint in Theorem 1.1. For instance, the braid word $w := aBAB$ satisfies the conditions of the theorem, but it is not geodesic: $[aBAB] = [aABA] = [BA].$ Here we write $a := \sigma_1, b:=\sigma_2, A:=\sigma_1^{-1}, B:=\sigma_2^{-1}$. The right statement is as follows. A freely reduced word $w$ in the alphabet $\{a,b,A,B\}$ is a geodesic in $B_3$ if and only if $w$ contains as subwords none of the following: elements of both $\{ab, ba\}$ and $\{AB,BA\}$; both $aba$ and $A$; both $aba$ and $B$; both $ABA$ and $a$; both $ABA$ and $b$; both $bab$ and $A$; both $bab$ and $B$; both $BAB$ and $a$; both $BAB$ and $b$. 
\end{Remark}

\begin{figure}[H]
\centering
\begin{tikzpicture} 
\braid[line width=1.5pt] s_1^{-1};
\end{tikzpicture} \hspace{1cm}
\begin{tikzpicture} 
\braid[line width=1.5pt] s_1;
\end{tikzpicture} \hspace{1cm}
\begin{tikzpicture} 
\braid[line width=1.5pt] s_1;
\braid[line width=1.5pt] s_1^{-1};
\end{tikzpicture}
\caption{Positive crossing (on the left), negative crossing (in the middle), and shadow crossing (on the right).}\label{E2}
\end{figure}
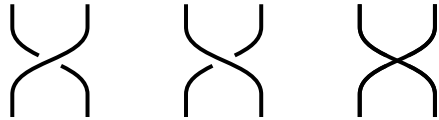

Let $B^3$ be the unit 3-dimensional ball in $\mathbb{R}^3$. An $n$-string tangle in $B^3$ is a collection $S$ of $n$ disjoint oriented intervals and some number of circles, properly embedded in $B^3$ in such a way, that the endpoints of each interval belong to the set $X=\{x_k\}_{k=1}^{2n}$, where $x_k$ are points on the great circle $z=0$ given by $x_k=(\exp(k\pi i/2n),0)$ for odd $k$ and $x_k=(\exp(-(k-1)\pi i/2n),0)$ for even $k$. 

Let $T$ be a tangle with $n$ oriented strings. We can distinguish two endpoints of a string: its input, and its output. Denote by $\partial^- T$ and $\partial^+ T$ sets of inputs and outputs of all strings of $T$, respectively. A tangle $T$ is ordered, if the set $X=\cup_{i=1}^{2n}x_i=\partial^- T\cup \partial^+ T$ is numbered by a collection $\{j_1<j_2<\dots<j_{2n}\}$ of integers, so that each input is numbered by $j_{2k-1}$, and each output is numbered by $j_{2k}$ for some $k=1,2\dots,n$. An ordering of a tangle $T$ is called coherent, if for every $k=1,2,\dots,n$ a string with the input numbered by $j_{2k-1}$ has its output ordered by $j_{2k}$.

By $D$ we denote a braid diagram and by $|D|$ we denote the number of crossings on $D$. By $|b|$ we denote the number of crossings on the minimal diagram of a braid $b \in B_n$. Given a braid diagram such as in Figure \ref{HomBraidExample}, we introduce the natural downward orientation on stings. Now we treat braid diagrams as tangles. 

Consider a crossing on a tangle $T$. We write $T = T_+$ in case the crossing is positive, and $T = T_-$ otherwise. Let $T_+$, $T_-$ and $T_0$ be tangles, which look as shown in Figure \ref{Skein} inside a disk and coincide outside this disk. The tangle $T_0$ is called the result of smoothing of $T_+$, $T_-$.

Following \cite{Polyak}, we recall the definitions of Alexander--Conway polynomial for ordered tangles. Let $D$ be an ordered tangle diagram. An $n$-state of $D$ is a collection of $n$ crossings of $D$. A state $S$ of $D$ defines a new tangle diagram $D(S)$, obtained from $D$ by smoothing all crossings of $S$. The smoothed diagram inherits an ordering from $D$. The state $S$ is called coherent, if $D(S)$ contains no closed components and the ordering of $D(S)$ is coherent, i.e. if both ends of each string numbered by $j_{2i-1}$ and $j_{2i}$ for some $i$. Suppose that $S$ is coherent. As one follows $D(S)$ along the first string of $D(S)$ (starting from its input and ending in its output), then continue to the second string of $D(S)$ in the same fashion, etc., one passes a neighborhood of each smoothed crossing $s\in S$ twice. A (coherent) state $S$ is called descending if one enters this neighborhood first time on the (former) overpass of $D$, and the second -- on the underpass. The sign of a positive crossing is defined to be $+1$ and the sign of a negative crossing is defined to be $-1$. The sign $\operatorname{sign}(S)$ of $S$ is defined as the product of signs of all crossings in $S$.

\begin{figure}[H]
\centering
\includegraphics[width = 8cm]{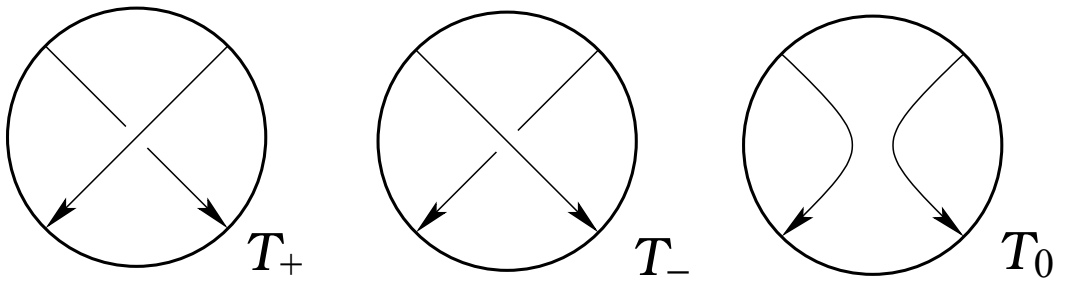}
\caption{Smoothing of a crossing on an oriented diagram. The picture was taken from \cite{Polyak}. }\label{Skein}
\end{figure}

\begin{Remark}\label{HomogeneousOrdering}
Let $D$ be a braid diagram with the standard ordering given by the natural numbering $x_k \mapsto k$. Consider a unique $|D|$--state $S$ of $D$. Note that $D(S)$ is the trivial diagram, so $S$ is coherent. Also, $S$ is descending if and only if $D$ is positive. Now let $D$ be a braid diagram on $n$ strings with an ordering by numbers $1,2, \ldots, 2n$. A unique $|D|$--state $S$ is coherent if and only if for all odd $k$ the output $x_{k+1}$ is numbered by $i+1$ whenever the input $x_k$ is numbered by $i$. If $S$ is descending, then $D$ is homogeneous.
\end{Remark}

An ordering of a braid diagram $D$ is said to be homogeneous if a unique $|D|$--state is descending. If the ordering is homogeneous and $D$ is non-degenerate, then there exist unique $e_1, e_2, \ldots, e_{n-1} \in \{1,-1\}$ such that $D$ is $(e_1,e_2,\ldots,e_{n-1})$--homogeneous. It is easy to check that for any homogeneous braid diagram $D$ there exists at least one homogeneous ordering.

For example, consider a homogeneous braid in Figure \ref{HomBraidExample}. There are more than one possible homogeneous orderings: one is presented in the figure, and another one may be obtained from that by changing numbers from $(19,17,20,18)$ to $(17,19,18,20)$.

Let $D$ be an ordered tangle diagram. Denote by $\mathcal{S}_m(D)$ the set of all descending $n$-states of $D$. Define $c_m(D)\in \mathbb{Z}$ and the Alexander--Conway polynomial $\nabla (D)\in\mathbb{Z}[z]$ of $D$ by 
\begin{align*}
a_m(D)=\sum_{S\in \mathcal{S}_m(D)}\operatorname{sign}(S)\,,\quad \nabla(D)=\sum_{m=0}^\infty a_m(D) z^m
\end{align*}
We emphasize that $\nabla(D)$ depends on an ordering. In fact, the coefficients $c_m(D)$ and the polynomial $\nabla(D)$ are the isotopy invariants, see Theorem 4.6. in \cite{Polyak}. Given a braid $b$ with an ordered diagram $D$, we denote by $\nabla(b)$ the Alexander--Conway polynomial of $D$.

\section{Homogeneous and alternating braids}

We assume that all link diagrams lie on the 2-sphere. Recall that a diagram of a link is called alternating if as we travel through the link diagram by any given orientation, the strings go through the crossings alternately between overpasses and underpasses. A link is called alternating if it has an alternating diagram. Given a link diagram, we consider its shadow. In the shadow, all positive and negative crossings turn out to be shadow crossings, see Figure \ref{E2}. A link diagram is called reduced if its shadow has the following property: no connected component of the shadow complement on the 2-sphere meets a double point from several sides at once. 
\begin{figure}[H]
\centering
\includegraphics[width = 6cm]{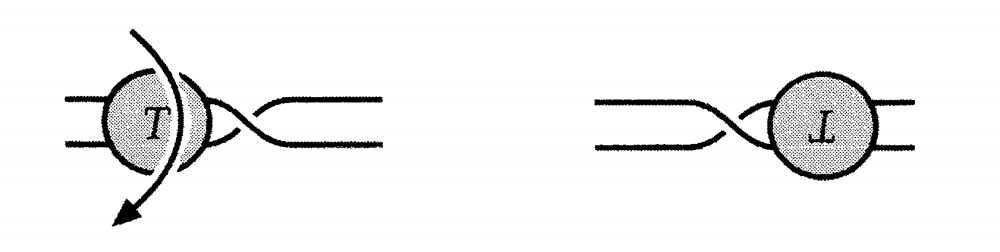}
\caption{Flypes. The picture was taken from \cite{Adams}.}\label{PF}
\end{figure}

We make several historical notes on alternating links. The fact that any two reduced alternating diagrams of the same link have the same number of crossings is one of the number of classical conjectures about alternating links that go back to the original compilations of knot tables by Tait and Little at the end of the 19th century. This result was proved in 1987 by Kauffman, Thistlethwaite, and Murasugi (see \cite{Kauffman, Thistlethwaite, M}) independently by using Jones polynomial for links, which is Laurent polynomial in one variable. It was shown that the span (that is, the difference between the maximal and minimal degrees) of the Jones polynomial $V_L(t)$ of a link $L$ is bounded from above by number of crossings on any diagram of $L$. Moreover, for any reduced alternating link diagram, the equality holds, and for any reduced, prime, non-alternating link diagram, the inequality is strict. The latter implies that any minimal diagram of a prime alternating link is alternating. Also, it is known that an alternating link $L$ is prime if and only if some reduced alternating diagram of $L$ is prime.

It is known that any two reduced alternating diagrams of the same link have the same writhe. Moreover, any two reduced alternating diagrams of the same link are related through a sequence of local moves, which are called flypes (see Figure \ref{PF}). D.W. Sumners has used these results to show that the number of knots grows at least exponentially as a function of minimal crossing number. Conjecture \ref{Wreathlikeconjecture} below, Conjecture \ref{HBGR}, and our lower bound on $v(B_n)$ in the introduction may be treated as braid analogues of these ideas.

\begin{figure}[H]
\centering
\includegraphics[width = 4cm]{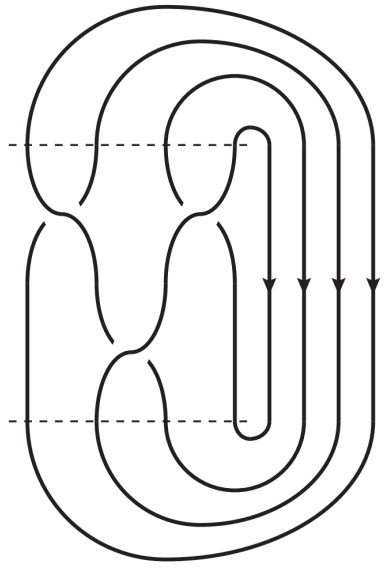} \hspace{2cm}
\includegraphics[width = 4cm]{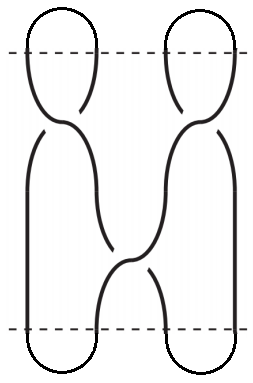}
\caption{The Alexander closure and the plat closure of a braid.}\label{PClosure}
\end{figure}

\subsection{Minimal diagrams of alternating braids}

A braid word $w$ in $B_n$ is called reduced if for any index $1 \leq k \leq n-1$ the total number of letters $\sigma_k,\sigma_k^{-1}$ containing in $w$ is either zero or at least two. Note that a braid diagram is reduced if and only if the corresponding braid word is reduced. Denote by $\overline{w}$ the Alexander closure of the braid diagram corresponding to $w$, see Figure \ref{PClosure}.

\begin{Remark}\label{Reduced}
A braid diagram is reduced if and only if its Alexander closure is reduced.
\end{Remark}

\begin{Remark}\label{AlternatingBraid}
A braid diagram is alternating if and only if its Alexander closure is alternating.
\end{Remark}

\begin{Theorem}\label{ProofOfConjectureForAlternatingBraids}
Suppose $w$ is a braid word representing an alternating braid. Then, $w$ is geodesic if and only if $w$ is alternating. Moreover, the monoids $\mathcal{HBM}(1,-1,\ldots,(-1)^n)$ and $\mathcal{HBM}(-1,1,\ldots,(-1)^{n-1})$ are isomorphic and have the following presentation
$$\langle a_1, a_2, \ldots, a_{n-1} \mid a_ia_j=a_ja_i, \ |i-j|\geq 2\rangle.$$
\end{Theorem}
\begin{proof}
Suppose $b \in B_n$ is an alternating braid on $n$ strings with an alternating braid word representative $w$. We prove that $w$ is geodesic. Note that $ww$ is reduced and alternating. Due to remarks \ref{AlternatingBraid} and \ref{Reduced}, we see that $\overline{ww}$ is a reduced alternating diagram of a link. It follows that $\overline{ww}$ is minimal. Hence $ww$ is minimal, so is $w$.

Now, suppose $b \in B_n$ is an alternating braid with a minimal braid word representative $w$. We prove that $w$ is alternating. Choose an alternating braid word representative $v$ of $b$. Without loss of generality, one can assume that $v$ is $(1,-1, \ldots, (-1)^n)$--homogeneous braid word. Take $u = (\sigma_1\sigma_2^{-1}\sigma_3\ldots \sigma_{n-1}^{(-1)^n})^{100}$. Note that $wu$ is a reduced diagram of a braid $b[u]$ and hence $D_1 = \overline{wu}$ is a reduced diagram of a link $L = \overline{b[u]}$ (see Remark \ref{Reduced}). Since $vu$ is an alternating braid word, $b[u]$ is an alternating braid, hence $D_2 = \overline{vu}$ is an alternating diagram of a link, so $L$ is alternating (see Remark \ref{AlternatingBraid}). It follows that $D_1$ is a minimal diagram. By the construction, $D_2$ is a prime diagram of $L$. It follows that $L$ is prime. Since $D_1$ is a minimal diagram of a prime link $L$, $D_1$ is alternating. Hence $wu$ is alternating braid diagram, so is $w$.

Now, suppose alternating braid words $v,w$ represent the same alternating braid in $B_n$. By adding a trivial string on the right, one can assume that $n$ is even. Without loss of generality, one can assume that $v,w$ are $(1,-1, \ldots, (-1)^n)$--homogeneous braid words. Take $u = (\sigma_2^{-1}\sigma_1\sigma_2^{-1}\sigma_3\ldots \sigma_{n-1}^{(-1)^n} \sigma_{n-2}^{(-1)^{n-1}})^{100}$, see Figure \ref{UsefulBraid}. 

Note that $uvu$, $uwu$ are reduced alternating diagrams of the same braid. Consider link diagrams obtained by plat closure of the braid words $uvu$ and $uwu$ as shown in Figure \ref{PClosure}. As for the Alexander closure, it is easy to check that these link diagrams are alternating. By the construction of $u$, they are reduced. It follows that there exist a sequence of flypes that transform the first diagram to the second. Note that no nontrivial flype may be performed. Hence these two diagrams are plane--isotopic. Hence $uvu$ and $uwu$ are plane--isotopic. This means that one can use the far commutativity and transform $uvu$ to $uwu$. Hence one can use the far commutativity to transform $v$ to $u$.
\end{proof}

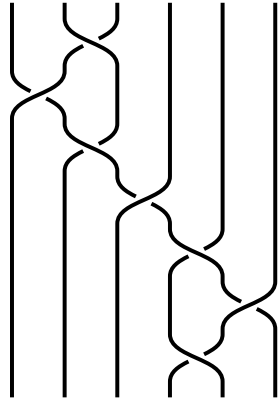
\begin{figure}[H]
\centering
\begin{tikzpicture}
\braid[scale=0.7,line width=1.5pt] s_2s_1^{-1}s_2s_3^{-1}s_4s_5^{-1}s_4;
\end{tikzpicture}
\caption{The braid word $\sigma_2^{-1}\sigma_1\sigma_2^{-1}\sigma_3\ldots \sigma_{n-1}^{(-1)^n} \sigma_{n-2}^{(-1)^{n-1}}$ for $n=6$.}\label{UsefulBraid}
\end{figure}

\subsection{Minimal diagrams of homogeneous braids}

It is well known that for any link $L$ there is a braid whose Alexander closure is $L$. Given a link $L$, by ${\rm Br}(L)$ we denote the least $n \in \mathbb{N}$ such that there exists $b \in B_n$ whose Alexander closure is $L$. The number ${\rm Br}(L)$ is called the braid index of $L$.

We aim to prove the classification theorem for minimal diagrams of homogeneous braids. For this, we use the Alexander-Conway polynomial for ordered tangles. Using the Alexander-Conway polynomial for links, one can prove the following result (see Proposition 7.4. in \cite{MurasugiA}).

\begin{Proposition}\label{MurasugiTheorem}
Let $w$ be a homogeneous braid word on $n$ strings. If ${\rm Br}(\overline{w}) = n$, then $\overline{w}$ is a minimal diagram of a link.
\end{Proposition}

The proof is based on the fact that for any diagram $D$ of a link $L$ the degree of the Alexander--Conway polynomial of $L$ is bounded from above by $|D|-s(D)+1$ (here $s(D)$ is the number of Seifert circles of $D$), and that for homogeneous link diagrams the equality holds.

The result above implies that any homogeneous braid diagram is minimal as follows. Let $w$ be a homogeneous braid word. Since the subgroup of pure braids has finite index in $B_n$, there exists $m \in \mathbb{N}$ such that the Alexander closure of $u := w^m$ is a link with $n$ components. In particular, ${\rm Br}(\overline{u}) = n$. Using Proposition \ref{MurasugiTheorem}, we see that $\overline{u}$ is a minimal link diagram. Hence $u = w^n$ is a geodesic word, so is $w$.

We did not manage to prove that any minimal diagram of a homogeneous braid is homogeneous by using only the Alexander closure technique. Nevertheless, we use the Alexander-Conway polynomial for tangles to prove the classification theorem.

\begin{Lemma}\label{Inequalities}
Let $D$ be an ordered diagram of a braid $b$. One has ${\rm deg} \nabla(b) \leq |b|$. The equality holds if and only if the ordering is homogeneous.
\end{Lemma}
\begin{proof}
Note that $a_{|D|+m} = 0$ for all $m \in \mathbb{N}$. Hence ${\rm deg} \nabla(D) \leq |D|$ and the inequality follows. From Remark \ref{HomogeneousOrdering} it follows that $a_{|D|} \neq 0$ if and only if the ordering of $D$ is homogeneous. Therefore, ${\rm deg} \nabla(D) = |D|$ if and only if the ordering of $D$ is homogeneous.
\end{proof}

\begin{Theorem}\label{HomogeneousAreGeodesic}
Any homogeneous braid diagram is minimal. Also, any minimal diagram of a homogeneous braid is homogeneous.
\end{Theorem}
\begin{proof}
Let $D$ be a homogeneous diagram of a braid $b \in B_n$. Take any homogeneous ordering of $D$. It follows from Lemma \ref{Inequalities} that
\begin{align*}
|D| \leq {\rm deg} \nabla(b) \leq |b| \leq |D|.    
\end{align*}
Hence $D$ is minimal.

Now, let $D$ be a minimal diagram of a homogeneous braid $b \in B_n$. Assume $D$ is not homogeneous. Then, $D$ has no homogeneous orderings, so $a_{|D|}=0$. The inequality ${\rm deg} \nabla(b)={\rm deg} \nabla(D) < |D| = |b|$ contradicts to Lemma \ref{Inequalities}.
\end{proof}

\begin{Remark}
Let $D$ be a $(e_1,e_2, \ldots, e_{n-1})$--homogeneous diagram of a braid $b$. The same arguments show that any minimal diagram of $b$ is $(e_1,e_2, \ldots, e_{n-1})$--homogeneous.
\end{Remark}

The remark above hints the following weak form of Conjecture \ref{HBGR}. Consider a $(e_1, e_2, \ldots, e_{n-1})$--homogeneous diagram $D$ of a braid on $n$ strings. Call two indexes $i,j \in \{1,2,\ldots,n-1\}$ equivalent, if for all $\min(i,j) \leq k,m \leq \max(i,j)$ one has $e_k = e_m$. These equivalence classes are said to be homogeneous blocks. For example, if $n=7$ and $(e_1, e_2, \ldots, e_{n-1}) = (1,1,-1,-1,1,1)$, then there are 3 equivalence classes: $\{1,2\}$, $\{3,4\}$, and $\{5,6\}$. In general, the equivalence classes are numbered in the natural way.

\begin{conjecture}\label{Wreathlikeconjecture}
Suppose $w_1,w_2$ are two homogeneous braid words representing the same braid. For $1 \leq k \leq n-1$, let $c_k$ (resp. $c_k^\prime$) be the number of letters $\sigma_i, \sigma_i^{-1}$ such that $i$ lies in the $k$-th homogeneous blocks of $w_1$ (resp. $w_2$). Then, $c_k = c_k^\prime$ for all $1 \leq k \leq n-1$.
\end{conjecture}

\section{Geodesic growth of braid groups with respect to Artin generators}

Consider a braid diagram $D$ on $n$ strings. Given $1 \leq k_1, \ldots, k_s \leq n$, consider a braid diagram on $s$ strings that is obtained from $D$ by deleting all strings except $k_1, k_2, \ldots, k_s$. All the resulting braid diagrams are called subdiagrams of $D$.

\begin{definition}
A braid diagram $D$ on $n$ strings is said to be winding if there exist $m \in \mathbb{N}$ and a multiset of braid subdiagrams $D_1, \ldots, D_s$ of $D$ with the following properties:
\begin{enumerate}
\item each $D_k$ is minimal;
\item for any two strings of $D$, there are exactly $m$ elements of the multiset that contain both of the strings.
\end{enumerate}
A braid word is said to be winding if the corresponding braid diagram is winding.
\end{definition}

\begin{Lemma}\label{WindingWordsAreGeodesic}
Any winding braid diagram is minimal.
\end{Lemma}
\begin{proof}
Let $D$ be a winding braid diagram with a multiset of subdiagrams $D_1, \ldots, D_s$ and constant $m$. Let $D^\prime$ be another diagram of the same braid. Denote by $D_1^\prime, D_2^\prime, \ldots, D_s^\prime$ the corresponding subdiagrams of $D^\prime$ on the same strings. Since $D_1, \ldots, D_s$ are minimal, one has $|D_i| \leq |D_i^\prime|$ for all $1 \leq i \leq s$. By counting the number of crossings on $D$ and $D^\prime$, one has
\begin{align*}
m |D| = \sum_{i = 1}^s |D_i| \leq \sum_{i=1}^s |D_i^\prime| \leq m |D^\prime|.
\end{align*}
Hence $|D| \leq |D^\prime|$, so $D$ is minimal.
\end{proof}
By the definition, any minimal braid diagram is winding with $s=1$, $D_1 = D$ and $m=1$. There are examples of minimal braid diagrams that satisfy winding diagram properties only in this trivial case (for example, minimal diagrams of Brunnian braids). The class of winding braid diagrams contains the following interesting class. A braid diagram on $n$ strings is said to be $k$--regular winding, $2\leq k \leq n$, if after deleting any $n-k$ strings on it the resulting subdiagram on $k$ strings turns to be minimal. For instance, any positive or negative braid diagram is $2$--regular winding, and generally speaking, $2$--regular winding braids are those whose any two strings $"$twist$"$ around each other in a positive or negative direction. That is, for any pair of strings $i<j,$ the crossings between $(i,j)$ are either all positive or all negative. It is easy to check that any $k$--regular winding diagram is also $m$--regular winding for all $2 \leq k \leq m \leq n$. Conjecture \ref{SMBC} holds for the braid groups $B_2, B_3$, hence the classes of $2$--regular and $3$--regular winding braid diagrams are closed under end extension. Therefore, Stallings' minimal braid diagrams conjecture holds for $2$--regular and $3$--regular winding braid diagrams.

\begin{Theorem}\label{GeodesicsGenerator}
For any $k \in \mathbb{N}$ there exist constants $c>0$, $N \in \mathbb{N}$ such that for all $m \geq N$ one has
\begin{align*}
\gamma_n(m) \geq c\cdot m^k \cdot (n-1)^m.
\end{align*}
\end{Theorem}
\begin{proof}
Given $k \in \mathbb{N}$, choose integers $0 \leq x_1, x_2, \ldots, x_{k+1}$. Consider arbitrary $k+1$ braid shadows that have $x_1, \ldots, x_{k+1}$ crossings respectively. Below we explain how to choose crossing types in that shadows in order to obtain corresponding braid words $v_1, \ldots, v_{k+1}$. 

Consider permutations $\pi_1, \ldots, \pi_{k+1} \in S_n$ defined by the braid diagrams $v_1, \ldots, v_{k+1}$. Given $1 \leq i \leq k+1$, consider a simple braid shadow on which the strings $(\pi_i(1), \pi_i(2))$ go to $(1,2)$, see Figure \ref{W}. By a simple braid shadow, we mean a braid shadow in which any two strings cross at most once. Below we explain how to choose crossing types in that $k+1$ shadows in order to obtain corresponding braid words $u_1, \ldots, u_{k+1}$.

We define a braid word $w = \sigma_1\sigma_2\sigma_3\ldots \sigma_{n-1}\sigma_{n-1}\sigma_{n-2} \ldots \sigma_2\sigma_1^{-1}\sigma_1^{-1}\sigma_2\sigma_3 \ldots \sigma_{n-1}\sigma_{n-1}\sigma_{n-2} \ldots \sigma_2\sigma_1$ of length $4n-4$. We put $W := v_1u_1w v_2u_2w v_3u_3w \ldots v_ku_kw v_{k+1}u_{k+1}.$
The braid word $W$ will turn out to be $3$--regular winding and hence geodesic by Lemma \ref{WindingWordsAreGeodesic}. We put $N := (k+1)(n-1)(n-2)/2 + k(4n-4) + (k+1)$. Let $m \geq N$ be a natural number. We assume that integers $x_1, \ldots, x_{k+1}$ satisfy $d + k(4n-4) + t = m,$ where $d := \sum_{i=1}^{k+1} x_i$ and $t := \sum_{i=1}^{k+1} |u_i|$.

\begin{figure}[H]
\centering
\begin{tikzpicture}
\braid[style strands={1}{red},style strands={2}{blue},scale=0.415,line width=1.5pt] s_1^{-1}s_2^{-1}s_3^{-1}s_4^{-1}s_5^{-1}s_5^{-1}s_4^{-1}s_3^{-1}s_2^{-1}s_1s_1s_2^{-1}s_3^{-1}s_4^{-1}s_5^{-1}s_5^{-1}s_4^{-1}s_3^{-1}s_2^{-1}s_1^{-1};
\braid[style strands={1}{red},style strands={2}{blue},scale=0.415,line width=1.5pt] s_1^{-1}s_2^{-1}s_3^{-1}s_4^{-1}s_5^{-1}s_5^{-1}s_4^{-1}s_3^{-1}s_2^{-1}s_1s_1s_2^{-1}s_3^{-1}s_4^{-1}s_5^{-1}s_5^{-1}s_4^{-1}s_3^{-1}s_2^{-1}s_1^{-1};
\end{tikzpicture} \hspace{2cm}
\begin{tikzpicture}
\braid[style strands={1}{red},style strands={2}{blue},scale=1,line width=1.5pt] s_1^{-1}s_2^{-1}s_2^{-1}s_1s_1s_2^{-1}s_2^{-1}s_1^{-1};
\braid[style strands={1}{red},style strands={2}{blue},scale=1,line width=1.5pt] s_1^{-1}s_2^{-1}s_2^{-1}s_1s_1s_2^{-1}s_2^{-1}s_1^{-1};
\end{tikzpicture} \hspace{2cm}
\begin{tikzpicture}
\braid[number of strands=7,
scale=1.135,line width=1.5pt] s_6s_5s_4-s_6s_3-s_5s_2-s_4s_1-s_3s_2;
\braid[number of strands=7,
scale=1.135,line width=1.5pt] s_6^{-1}s_5^{-1}s_4^{-1}-s_6^{-1}s_3^{-1}-s_5^{-1}s_2^{-1}-s_4^{-1}s_1^{-1}-s_3^{-1}s_2^{-1};
\end{tikzpicture} 
\caption{The braid diagram corresponding to $w$ (on the left), the braid diagram corresponding to $\hat{w} = abbAAbba$ (in the middle) and a simple shadow (on the right).
}\label{W}
\end{figure}
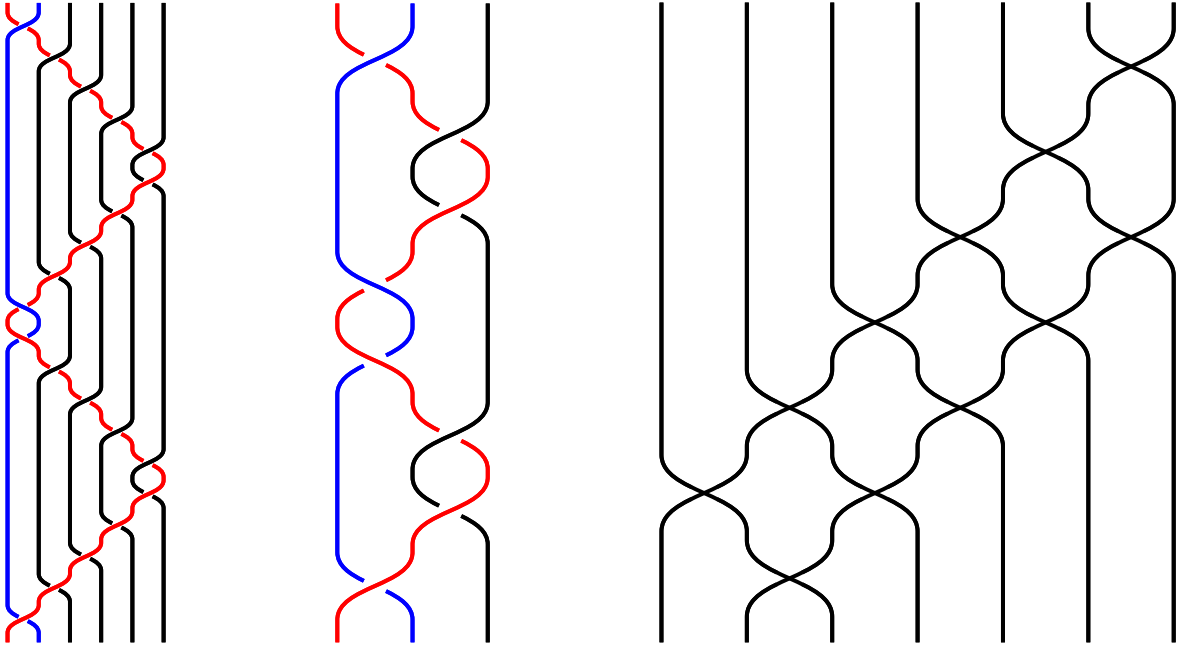

By the construction, there are $(n-1)^{x_i}$ different possibilities of choosing each shadow $v_i$. Hence there are $(n-1)^d$ possibilities of choosing $W$ with given $x_1, \ldots, x_{k+1}$. Note that there are 
$$((m-t)-(4n-4))\cdot ((m-t)-2(4n-4))\cdot  \ldots \cdot ((m-t)-k(4n-4))/k!$$ 
possibilities of choosing $x_1, \ldots, x_{k+1}$, that is, possibilities for positions of $w$'s. This gives
\begin{align*}
\gamma_n(m) &\geq \dfrac{((m-t)-(4n-4))\cdot ((m-t)-2(4n-4))\cdot  \ldots \cdot ((m-t)-k(4n-4))}{k!} \cdot (n-1)^{d} \\
& \geq \dfrac{((m-t)-(4n-4))^k}{k!} \cdot (n-1)^{d} \geq c^\prime \cdot \dfrac{m^k}{k!} \cdot (n-1)^{d} = c \cdot m^k \cdot (n-1)^m,
\end{align*}
where the constant $c>0$ depends only on $n,k$. 

It remains to define the crossings in $v_i, u_i$ and prove that $W$ is $3$--regular winding. Each $v_iu_i$ is going to be $2$--regular winding. More precisely, we put all crossings in $v_iu_i$ between the strings $1,2$ to be positive. For each string $r \neq 1,2$ we put all crossings in $W$ between $1,r$ to be positive and all crossings between $2,r$ to be negative. Finally, for each $\{i,j\} \cap \{1,2\} = \varnothing$ we put the crossings between strings $i,j$ in $W$ to be all positive (actually, it does not matter). 

Now, we prove that $W$ is $3$--regular winding. Consider a subdiagram of the braid diagram of $W$ on three strings $i,j,r$. If $\{i,j,r\} \neq \{1,2,r^\prime\}$, then the subdiagram is $2$--regular winding and hence minimal. It remains to prove that the subdiagrams on $1,2,r$ are minimal for all $3 \leq r \leq n$. 

Let us check the minimality conditions from Remark \ref{B3} for each of these subdiagrams. Denote by $\hat{W}$ the braid word that corresponds to the subdiagram of $W$ on three strings $1,2,r$. The braid word $\hat{w}$ is defined in a similar way. Since the crossing type in $W$ between $1,k$ is positive, whereas the crossing type in $W$ between $2,k$ is negative, we see that $\hat{W}$ contains neither $aba$, $bab$, $ABA$, nor $BAB$. Assume $\hat{W}$ contains elements of both $\{ab,ba\}$ and $\{AB,BA\}$. Note that $\hat{w}$ contains neither $AB$ nor $BA$, but $\hat{w}$ contains $ab$ and $ba$. It remains to show that $\hat{W}$ does not contain neither $AB$ nor $BA$. Note that in each of the diagrams corresponding to $AB$, $BA$ there is a string that crosses two other strings negatively.
By the construction, there is no such string in $\hat{W}$. This shows that $\hat{W}$ is geodesic. Hence $W$ is $3$--regular winding. This completes the proof.
\end{proof}

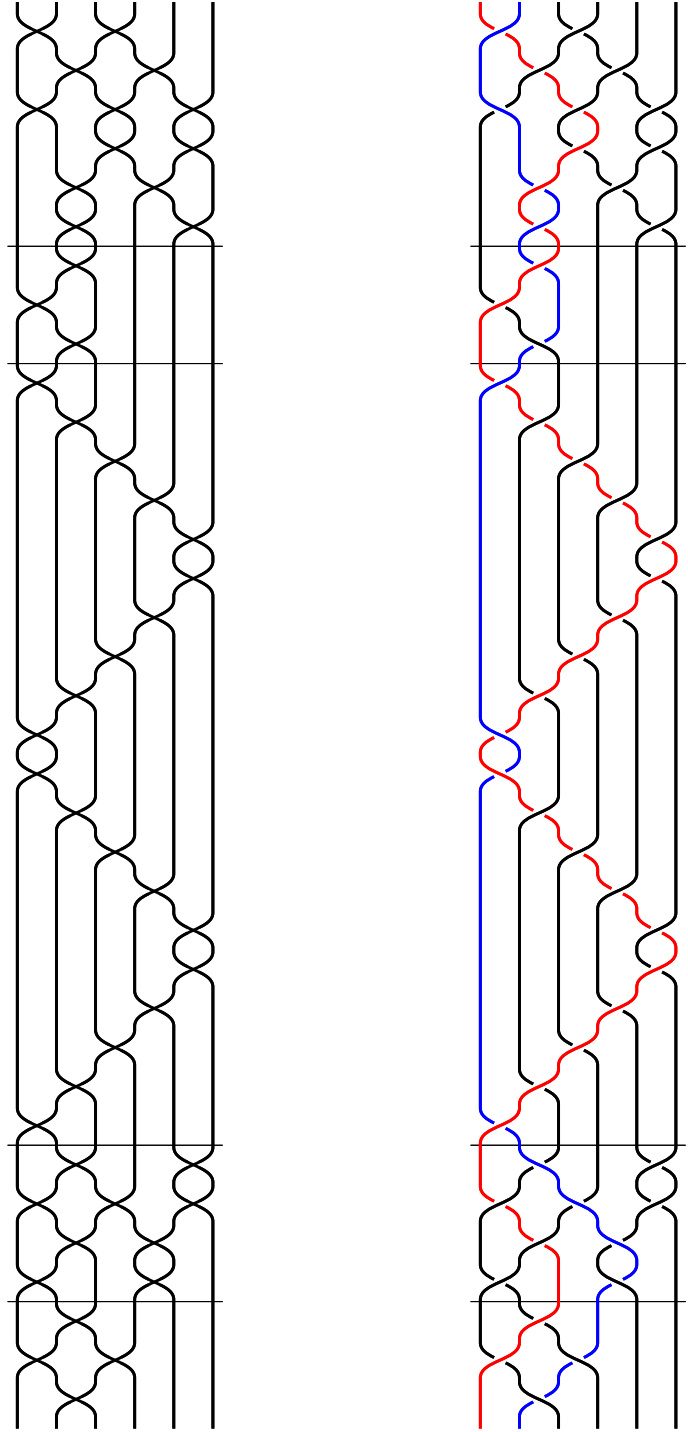
\begin{figure}[H]
\centering
\begin{tikzpicture}
\braid[scale=0.52, line width=1.1pt,floor command={
 \draw (\floorsx,\floorsy) -- (\floorex,\floorsy);
 },] 
s_1^{-1}-s_3^{-1}s_2^{-1}-s_4^{-1}s_1-s_3^{-1}-s_5^{-1}s_3^{-1}-s_5^{-1}s_2^{-1}-s_4^{-1}s_2^{-1}-s_5^{-1}|s_2^{-1}s_1^{-1}s_2|s_1^{-1}s_2^{-1}s_3^{-1}s_4^{-1}s_5^{-1}s_5^{-1}s_4^{-1}s_3^{-1}s_2^{-1}s_1s_1s_2^{-1}s_3^{-1}s_4^{-1}s_5^{-1}s_5^{-1}s_4^{-1}s_3^{-1}s_2^{-1}s_1^{-1}|s_2-s_5^{-1}s_1^{-1}-s_3-s_5^{-1}s_2^{-1}-s_4s_1^{-1}-s_4|s_2^{-1}s_1^{-1}-s_3s_2;
\braid[scale=0.52, line width=1.1pt,floor command={
 \draw (\floorsx,\floorsy) -- (\floorex,\floorsy);
 },] 
s_1-s_3s_2-s_4s_1^{-1}-s_3-s_5s_3-s_5^{1}s_2^{1}-s_4^{1}s_2^{1}-s_5^{1}|s_2^{1}s_1^{1}s_2^{-1}|s_1^{1}s_2^{1}s_3^{1}s_4^{1}s_5^{1}s_5^{1}s_4^{1}s_3^{1}s_2^{1}s_1^{-1}s_1^{-1}s_2^{1}s_3^{1}s_4^{1}s_5^{1}s_5^{1}s_4^{1}s_3^{1}s_2^{1}s_1^{1}|s_2^{-1}-s_5^{1}s_1^{1}-s_3^{-1}-s_5^{1}s_2^{1}-s_4^{-1}s_1^{1}-s_4^{-1}|s_2^{1}s_1^{1}-s_3^{-1}s_2^{-1};
\end{tikzpicture} \hspace{3cm}
\begin{tikzpicture}
\braid[scale=0.52, style strands={1}{red},style strands={2}{blue}, style strands={3}{black},style strands={4}{black}, line width=1.1pt,floor command={
 \draw (\floorsx,\floorsy) -- (\floorex,\floorsy);
 },] 
s_1^{-1}-s_3^{-1}s_2^{-1}-s_4^{-1}s_1-s_3^{-1}-s_5^{-1}s_3^{-1}-s_5^{-1}s_2^{-1}-s_4^{-1}s_2^{-1}-s_5^{-1}|s_2^{-1}s_1^{-1}s_2|s_1^{-1}s_2^{-1}s_3^{-1}s_4^{-1}s_5^{-1}s_5^{-1}s_4^{-1}s_3^{-1}s_2^{-1}s_1s_1s_2^{-1}s_3^{-1}s_4^{-1}s_5^{-1}s_5^{-1}s_4^{-1}s_3^{-1}s_2^{-1}s_1^{-1}|s_2-s_5^{-1}s_1^{-1}-s_3-s_5^{-1}s_2^{-1}-s_4s_1^{-1}-s_4|s_2^{-1}s_1^{-1}-s_3s_2;
\braid[scale=0.52, style strands={1}{red},style strands={2}{blue}, style strands={3}{black},style strands={4}{black}, line width=1.1pt,floor command={
 \draw (\floorsx,\floorsy) -- (\floorex,\floorsy);
 },] 
s_1^{-1}-s_3^{-1}s_2^{-1}-s_4^{-1}s_1-s_3^{-1}-s_5^{-1}s_3^{-1}-s_5^{-1}s_2^{-1}-s_4^{-1}s_2^{-1}-s_5^{-1}|s_2^{-1}s_1^{-1}s_2|s_1^{-1}s_2^{-1}s_3^{-1}s_4^{-1}s_5^{-1}s_5^{-1}s_4^{-1}s_3^{-1}s_2^{-1}s_1s_1s_2^{-1}s_3^{-1}s_4^{-1}s_5^{-1}s_5^{-1}s_4^{-1}s_3^{-1}s_2^{-1}s_1^{-1}|s_2-s_5^{-1}s_1^{-1}-s_3-s_5^{-1}s_2^{-1}-s_4s_1^{-1}-s_4|s_2^{-1}s_1^{-1}-s_3s_2;
\end{tikzpicture}
\caption{A braid word produced by the construction: $k=1$, $x_1 = 13,$ and $x_2=9$.}\label{Example}
\end{figure}

\section{Some remarks on Stallings' minimal braid diagrams conjecture}

Let $G = \langle S \rangle$ be a group with a generating set $S$. Given $g \in G,$ denote by $R(g)$ the set of letters $s \in S\cup S^{-1}$ such that there exists a geodesic word $ws \in F(S)$ with the property $[ws] = b$. By $l_{G,S}(g)$ we denote the word--metric distance between $g \in G$ and $1 \in G$. For example, the length of a braid in the braid group (with the standard Artin generators) is the crossing number of a braid.

Assume that a presentation $G = \langle S \mid R\rangle$ has relations of even length, that is, for all $r \in R$ the length $|r|$ is even. It is easy to see that $l_{G,S}(g[s]) = l_{G,S}(g) \pm 1$ for all $s \in S\cup S^{-1}$. In particular, given a geodesic word $w \in F(S)$ and a letter $s \in S\cup S^{-1}$, $ws$ is geodesic if and only if $s^{-1} \notin R([w])$ (see Lemma 3.1. in \cite{IsobelWebster}). It follows that $R(g) = S \cup S^{-1}$ if and only if $g$ is a dead end element (see \cite{OH}), that is, $l_{G,S}(gs) \leq l_{G,S}(g)$ for all $s \in S \cup S^{-1}$. It is easy to check that the following statement holds.

\begin{Lemma}\label{eq}
Assume a presentation $G = \langle S \mid R \rangle$ has relations of even length. The following conditions are equivalent:
\begin{enumerate}
\item For any geodesic word of the form $ws$ ($s \in S \cup S^{-1}$) the word $wss$ is geodesic;
\item For all $g \in G$, the following holds true: if $s \in R(g)$ then $s^{-1} \notin R(g)$;
\item Let $w$ be a geodesic word. For all $s \in S\cup S^{-1}$ there exists $e \in \{1,-1\}$ such that $ws^e$ is geodesic.
\end{enumerate}
\end{Lemma}

Note that the standard presentation for $B_n$ with the Artin generators has relations of even length. For $n\geq 4$, it is unknown whether there are dead end elements in the braid groups $B_n$. If Stallings' minimal braid diagrams conjecture is true, then one has $|R(b)| \leq n-1$ for any $b \in B_n$. Actually, this ineqaulity is sharp, as we explain below. Recall that $B_n^+ = \mathcal{HBM}(1, 1, \ldots, 1)$ is the monoid of positive braids. Following Definition 9.1.12 in \cite{WPG}, we denote by $\succ$ the natural order on $B_n^+$ given by $x \succ y$ whether $zy = x$ for some $z \in B_n^+$. It is known that $(B_n^+, \succ)$ is a lattice, that is, any two elements of $B_n^+$ have a unique supremum. The braid $\Delta = (\sigma_1\sigma_2\ldots \sigma_{n-1})(\sigma_1\sigma_2\ldots \sigma_{n-2}) \ldots (\sigma_1\sigma_2)\sigma_1$ is called the fundamental (Garside) element in $B_n$.
Note that $R(\Delta) = \{\sigma_1, \ldots, \sigma_{n-1}\}$. 
Moreover, a positive braid $b \in B_n^+$ satisfies $|R(b)| = n-1$ if and only if $b \succ \Delta$ (see Theorem 3 in \cite{Garside}). By considering divisors of the fundamental element and by changing their crossing types, one can prove that for any subset $T \subset \{\sigma_1, \ldots, \sigma_{n-1}, \sigma_1^{-1}, \ldots, \sigma_{n-1}^{-1}\}$ without inverse elements there exists $b \in B_n$ such that $R(b) = T$.


\begin{thebibliography}{99}
\bibitem{K} R. Kirby (ed.), {\it Problems in low-dimensional topology}, Geometric topology (Athens, GA, 1993) AMS/IP Stud. Adv. Math., vol. 2, Amer. Math. Soc., Providence, RI, 1997, pp. 35–473. MR 1470751.
\bibitem{Stallings} J. Stallings, {\it Constructions of fibred knots and links}, Algebraic and geometric topology (Proc. Sympos. Pure Math., XXXII, Amer. Math. Soc.), 2:55–60, 1978.
\bibitem{Cromwell} P. R. Cromwell, {\it Homogeneous links}, Journal of the London Mathematical Society, 2(3):535–552, 1989.
\bibitem{Stoimenow} A. Stoimenow, {\it On the crossing number of positive knots and braids and braid index criteria of Jones and MortonWilliams-Franks}, Trans. Amer. Math. Soc. 354(10) (2002), 3927–3954
\bibitem{AS} A. Stoimenow, {\it Non-triviality of the Jones polynomial and the crossing numbers of amphicheiral knots}, 2007, arXiv:math/0606255v2.
\bibitem{Sabalka} L. Sabalka, {\it Geodesics in the braid group on three strands}, Group theory, statistics, and cryptography, 133–150, Contemp. Math., 360, Amer. Math. Soc., Providence, RI, 2004.
\bibitem{Berger} M. Berger, {\it Minimum crossing numbers for 3--braids}, J. Phys. A 27 (1994), no. 18, 6205-6213.
\bibitem{Adams} C. Adams, {\it The Knot Book}, American Mathematical Soc., 1994.
\bibitem{Turaev1990} V. G. Turaev, {\it Jones-type invariants of tangles}, Journal of Mathematical Sciences 52 (1990) 2806–2807.
\bibitem{Polyak} M. Polyak, {\it Alexander-Conway invariants of tangles}, arXiv:1011.6200.
\bibitem{MurasugiB} K. Murasugi, B. I. Kurpita, {\it A study of braids}, Springer, 1999.
\bibitem{OH} M. Clay, D. Margalit, {\it Office Hours with a Geometric Group Theorist}. Princeton University Press, 2017.
\bibitem{Garside}F. Garside, {\it The braid groups and other groups}, Quart. J. Math. (Oxford) 20 (1969), 235–254.
\bibitem{Deligne} P. Deligne, {\it Les immeubles des groupes de tresses generalises}, Invent. Math. 17 (1972), 273--302.
\bibitem{Xu} P. Xu, {\it Growth of the positive braid semigroups}, J. Pure Appl. Algebra 80 (1992), 197--215.
\bibitem{CharneyMeier} R. Charney, J. Meier, {\it The language of geodesics for garside groups}, 2003.
\bibitem{WPG} D. B. A. Epstein, J. W. Cannon, D. F. Holt, S. V. F. Levy, M. S. Paterson, and W. P. Thurston, {\it Word processing in groups}, Jones and Bartlett Publishers, Boston, MA, 1992.
\bibitem{GarsideBook} P. Dehornoy, F. Digne, E. Godelle, D. Krammer, J. Michel, {\it Foundations of Garside Theory}, arXiv:1309.0796.
\bibitem{Mairesse} J. Mairesse, F. Matheus, {\it Growth series for Artin groups of dihedral type}, 2005.
\bibitem{GeodesicProblems} M. Elder, A. Rechnitzer, {\it Some geodesic problems in groups}, Groups. Complexity. Cryptology, 2(2):223–229, 2010.
\bibitem{FourStrandsAlgorithm} J. Mairesse, {\it Minimizing braids on four strands}, Joint work with A. Micheli, LIAFA, and D. Poulalhon, LIX, 2011.
\bibitem{Kauffman} L. Kauffman, {\it State models and the Jones polynomial}, Topology 26(3) (1987), 395–407.
\bibitem{Thistlethwaite} M. Thistlethwaite, {\it A Spanning Tree Expansion of the Jones Polynomial}, Topology 26(3) (1987), 297–309.
\bibitem{M} K. Murasugi, {\it Jones polynomials and classical conjectures in knot theory. II}, Mathematical Proceedings of the Cambridge Philosophical Society. 102 (2): 317–318, 1987.
\bibitem{MurasugiA} K. Murasugi, {\it On the braid index of alternating links}, Trans. Amer. Math. Soc. 326 (1) (1991), 237–260.
\bibitem{Vershik} A. M Vershik, S. Nechaev, R. Bikbov, {\it Statistical properties of locally free groups with applications to braid groups and growth of random heaps}, Commun. Math. Phys. 212 (2000), 469–501.
\bibitem{Dehornoy} P. Dehornoy, {\it Groupes de Garside}. Ann. Sci. Ecole Norm. Sup. (4), 35(2):267–306, 2002.
\bibitem{Brazil} M. Brazil, {\it Monoid growth functions for braid groups}, Int. J. of Algebra and Comp., 1(4):201–205, 1991.
\bibitem{Saito} K. Saito, {\it Growth functions for Artin monoids}, Proc. Japan Acd., 85, Ser. A (2009), pp84-88.
\bibitem{IsobelWebster} I. Webster, {\it Presentations of Groups with Even Length Relations}, arXiv:1909.10494.
\bibitem{Charney} R. Charney, {\it An introduction to right--angled Artin groups}, Geom. Dedicata 125 (2007), 141-158.
\bibitem{Charney1} R. Charney, {\it Artin groups of finite type are biautomatic}. Math. Ann., 292(4):671–684, 1992.
\bibitem{Charney2} R. Charney, {\it Geodesic automation and growth functions for Artin groups of finite type}. Math. Ann., 301(2):307–324, 1995.
\end{thebibliography}
\end{document}